\documentclass[runningheads,envcountsame,a4paper]{llncs} 

\usepackage{amssymb,amsmath}
\usepackage{wasysym}
\usepackage{graphicx}
\usepackage{epsfig}
\usepackage{latexsym}
\usepackage[english]{babel}
\usepackage[utf8]{inputenc}
\usepackage{csquotes}
\usepackage{algorithm} 
\usepackage{algpseudocode}
\usepackage{tikz}
\usetikzlibrary{arrows,shapes.geometric,positioning,calc}
\usepackage[colorinlistoftodos]{todonotes}
\usepackage{changepage}
\usepackage{subfig}
\usepackage{soul}
\usepackage{hyperref}
\usepackage{todonotes}
\usepackage{enumerate}

\newtheorem{thm}{Theorem}
\newtheorem{prop}[thm]{Proposition}
\newtheorem{cor}[thm]{Corollary}
\newtheorem{lem}[thm]{Lemma}

\newtheorem{rmk}[thm]{Remark}
\newtheorem{qst}[thm]{Question}

\newsavebox{\qedB}

\sbox{\qedB}{\setlength{\unitlength}{1mm}
 \begin{picture}(4,4)(0,0)
  \thinlines
  {\put(0,0){\framebox(2.83,2.83){}}}%
  {\put(1.17,1.17){\framebox(2.83,2.83){}}}%
  {\put(0,0){\framebox(4,4){}}}%
  {\put(1.17,1.17){{\rule{1ex}{1ex} }}}%
 \end{picture}}
 
\newcommand{\QEDB}{\ifmmode\def\next{\tag"\usebox{\qedB}"}%
 \else\let\next=\relax
 {\unskip\nobreak\hfil\penalty50
 \hskip2em\hbox{}\nobreak\hfil\usebox{\qedB}
 \parfillskip=0pt \finalhyphendemerits=0\penalty-100\bigskip}\fi\next}

\newcommand{\bprop}{\begin{prop}}
\newcommand{\eprop}{\end{prop}}
\newcommand{\bcor}{\begin{cor}}
\newcommand{\ecor}{\end{cor}}
\newcommand{\blem}{\begin{lem}}
\newcommand{\elem}{\end{lem}}

\title{Letter frequency vs factor frequency in pure morphic words}

\authorrunning{}
\tocauthor{}
\author{Shuo Li}

\institute{\institute{Department of Mathematics \& Statistics, The University of Winnipeg,\\ Winnipeg, Canada.
}
}

\begin{document}

\maketitle


\begin{abstract}

We prove that, for any pure morphic word $w$, if the frequencies of all letters in $w$ exist, then the frequencies of all factors in $w$ exist as well. This result answers a question of Saari in his doctoral thesis.
\end{abstract}

\section{Introduction}
Let $A$ be a finite alphabet and let $w=w[1]w[2]w[3]\cdots$ be an infinite word over $A$. Let $v$ be a factor of $w$ of length $k$, the {\em (ordinary) frequency} of $v$ in $w$ is defined as $$\lim_{n \to \infty}\frac{\#\{i|\; w[i]w[i+1]\cdots w[i+k-1]=v, 1 \leq i \leq n \}}{n},$$
if the limit exists. For the factors such that $k=1$, the previous frequency is called the {\em (ordinary) frequency of letters}. The frequency of letters in morphic words has been studied since the early 1970’s. It was proved by Cobham in 1972 \cite{Cobham} that the frequency of a letter in an automatic word, if it exists, is rational. Michel \cite{michel,Michel1} proved that the frequencies of all letters exist in primitive morphic words. Peter \cite{PETER03} gave a sufficient and necessary condition for the existence of the frequency of a letter in an automatic word. Saari proved in \cite{Saari1} that the frequencies of both letters exist in any pure morphic binary word and gave a sufficient and necessary condition for the existence of the frequency of a letter in a morphic sequence over an arbitrary alphabet \cite{Saari2}. For any infinite word $w$ and any factor $v$ of length $k$ of $w$, one can also define the {\em logarithmic frequency} of $v$ in $w$ as 
$$\lim_{n \to \infty}\frac{1}{\log(n)}\sum_{\{i| \; w[i]w[i+1]\cdots w[i+k-1]=v, 1 \leq i \leq n \}}\frac{1}{i},$$
if this limit exists. It is proved in \cite[Proposition 8.4.4]{allouche_shallit_2003} that if the frequency of a letter $a$ in $w$ exists, then the logarithmic frequency of $a$ in $w$ exists as well, and these two frequencies are equal. Allouche and Shallit asked in \cite[Section 8.8]{allouche_shallit_2003} if the logarithmic frequency of any letter in a morphic word must exist and Bell gave a positive answer to this question in \cite{Bell08} and generalized it to all factors. Concerning the relation between the frequency of letters and the frequency of factors, Saari proved in his doctoral thesis \cite[Proposition 3.1]{Saari3} that, for an arbitrary morphic word $w$, the existence of the frequencies of all letters in $w$ cannot imply the existence of the frequencies of all factors in $w$ and asked if it is true for pure morphic words \cite[Problem 3.1]{Saari3}. We establish a positive answer to this question.

\begin{thm}
\label{main}
For any pure morphic word $w$, if the frequencies of all letters in $w$ exist, then the frequencies of all factors in $w$ exist as well.
\end{thm}

\section{Definitions and notation}

Let $A$ be a finite set. It will be called an {\em alphabet} and its elements will be called {\em letters}. Let $A^*$ denote the free monoid generated by $A$ under concatenations having neutral element the empty word $\varepsilon$. The elements in $A^*$ are called the {\em finite words} with letters in $A$. For any finite word $w=w[1]w[2]w[3]\cdots w[n] \in A^*$, the \emph{length} of $w$ is the integer $|w| = n$. Let $A^{\mathbf{N}}$ be the set of infinite concatenations of elements in $A$. The elements in $A^{\mathbf{N}}$ are called the {\em infinite words} with letters in $A$. For any infinite word $w$, the length of $w$, which is also denoted by $|w|$, is infinite. Let $A^{\infty}=A^* \cup A^{\mathbf{N}}$. For any $w\in A^{\infty}$ and $v \in A^{*}$, let $|w|_v$ denote the number of occurrences of $v$ in $w$. 

Let $w=w[1]w[2]w[3]\cdots$ be an element in $A^{\infty}$ and let $v \in A^*$. We say $v$ a {\em prefix} of $w$ if there exists an integer $t$ such that $1 \leq t \leq |w|$ and $v=w[1]w[2]w[3]\cdots w[t]$ and we say $v$ a {\em factor} of $w$ if there exists a pair of integers $t, r$ such that $1 \leq t \leq r \leq |w|$ and $v=w[t]w[t+1]w[t+2]\cdots w[r]$. A prefix $v$ of $w$ is called {\em proper} if $|v| <|w|$. For any pair of integers $t, r$ such that $0 \leq t \leq r \leq |w|$, let $w[t,r]=w[t]w[t+1]w[t+2]\cdots w[r]$.

Let $A$ and $B$ be two alphabets. A \emph{morphism} $\phi$ is a map $A^* \to B^*$ satisfying $\phi(xy)=\phi(x)\phi(y)$ for any pair of elements $x, y$ in $A^*$. The morphism $\phi$ is called \emph{$k$-uniform} for some positive integer $k$ if for all elements $a \in A$, $|\phi(a)|=k$, and it is called \emph{non-uniform} otherwise. A morphism $\phi$ is called a \emph{coding} function if it is $1$-uniform, and it is called \emph{non-erasing} if $\phi(a) \neq \varepsilon$ for all $a \in A$. For any positive integer $k$, by $\phi^k$, we mean the $k$-fold composition of the morphism $\phi$. Let $u \in A^*$ and $v \in B^*$, if $\phi(u)=v$, then $v$ is called the {\em image} of $u$ and $u$ is the {\em pre-image} of $v$ under $\phi$. We write $u=\phi^{-1}(v)$. A morphism $\phi: A^* \to A^*$ is called \emph{primitive} if there exists an integer $n \geq 1$ such that for all $a, b \in A$, $a$ occurs in $\phi^n(b)$.

Let $A$ be an alphabet and let $\phi: A^* \to A^*$ be a morphism. A letter $a \in A$ is called {\em bounded} if $|\phi^k(a)|$ is upper bounded by some constant $C$ for all positive integers $k$ and it is called {\em unbounded} otherwise.

Let $A$ be a finite alphabet, and let $w$ be an infinite word over $A$. $w$ is called \emph{morphic} if there exists an alphabet $B$, a letter $ b \in B$, a nonempty word $v \in B^*$, a non-erasing morphism $\phi: B^{*} \to B^{*}$, and a coding function $\psi: B \to A$, such that $\phi(b)=bv$ and $$w= \lim_{i \to \infty}\psi(\phi^i(b)).$$ Moreover, for any positive integer $k \geq 2$, the word $w$ is called \emph{$k$-automatic} if $\phi$ is $k$-uniform, it is called \emph{automatic} if $\phi$ is $k$-uniform for some integer $k \geq 2$, and it is called \emph{non-automatic} if $\phi$ is not $k$-uniform for any integer $k \geq 2$. The word $w$ is called {\em pure morphic} if $A=B$ and $\psi=Id$ and it is called {\em primitive} if $\phi$ is primitive. Remark that if $w$ is pure morphic, then $w=\phi(w).$

\section{Proof of the main theorem}

Let $A$ be an alphabet, $\phi: A^* \to A^*$ be a non-erasing morphism and $w$ be a pure morphic word such that $w=\phi(w)$. Let $A_B=\{a|\; a \in A,\; \text{$a$ is bounded}\}$ and let $A_U=\{a|\; a \in A,\; \text{$a$ is unbounded}\}$. Obviously, $A_B \cup A_U=A$ and $A_B \cap A_U=\emptyset$. If all letters in $A$ have a frequency, then for any $a \in A$, let $\alpha_a=\lim_{n \to \infty}\frac{|w[1,n]|_a}{n}$ and let $\alpha=\sum_{a \in A_B} \alpha_a$. Let $k_1$ be an integer such that $|\phi^n(a)|<k_1$ for all $a \in A_B$ and for all $n \in \mathbb{N}^+$.

The idea of the proof is that, for any pure morphic word $w$, any factor $v$ of $w$ and any prefix $w[1,m]$ of $w$ with $m$ large enough, supposing that there exists the word $w[1,m'']=\phi^{-k}(w[1,m])$ for some positive integer $k$, then one can prove that, when $k$ is large enough, $|w[1,m]|_v$ can be approximated by the summation of $|\phi^k(w[i])|_v$ for all unbounded letters $w[i]$, $1 \leq i \leq m''$. Thus, the frequency of $v$ can be estimated in terms of $\alpha_a$, $|\phi^k(a)|$ and $|\phi^k(a)|_v$, for all $a \in A_U$.

The strategy of the proof is as follow: supposing that there exists some factor $v$ of $w$ such that $\limsup\frac{|w[1,n]|_v}{n}-\liminf\frac{|w[1,n]|_v}{n}= \Delta>0$, then one can find a real number $C$, independent from $\Delta$, such that $\left|\frac{|w[1,n]|_v}{n}-C\right|<\frac{1}{3}\Delta$ when $n$ is large, which leads to a contradiction.

\begin{proof}[of Theorem \ref{main}]
Let $w$ be a pure morphic word satisfying the hypothesis in Theorem \ref{main} and suppose that there exists a factor $v$ of $w$ of length $L$ such that $\limsup\frac{|w[1,n]|_v}{n}-\liminf\frac{|w[1,n]|_v}{n}= \Delta>0$. Let $M \in \mathbb{N}$ such that for all $a \in A_U$
$$|\phi^M(a)| \geq \max\left\{\frac{6(L+\alpha k_1)}{(1-\alpha)\Delta}, \frac{k_1 \alpha(6-\Delta)}{(1-\alpha)\Delta}\right\}.$$
Let $k_2$ be an integer such that $|\phi^M(a)|<k_2$ for all $a \in A$.

For any $m \in \mathbb{N}^+$, let $m'$ be the largest integer smaller than or equal to $m$ such that $w[1,m']$ is the image of a prefix of $w$ under $\phi^M$. Thus, there exists $m'' \in \mathbb{N}^+$ such that
\begin{equation}\label{eq1}w[1,m']=\phi^M(w[1,m'']). \end{equation}
From the definition of $m'$, if $w[m'+1,m]$ is not empty, then it is a proper prefix of $\phi^M(a)$ for some $a \in A$. Thus,
\begin{equation}\label{eq2}m' \leq m < m'+k_2, \end{equation}
\begin{equation}\label{eq3}|w[1,m']|_v \leq |w[1,m]|_v < |w[1,m']|_v+k_2. \end{equation}
One can estimate $m'$ as well as $|w[1,m']|_v$ in terms of $m''$.
\begin{equation}\label{eq4}m' =\sum_{a \in A_U}|\phi^M(a)|\cdot|w[1,m'']|_a+\epsilon_1, \end{equation}
where 
\begin{equation}\label{eq5} \epsilon_1=\sum_{a \in A_B}|\phi^M(a)|\cdot|w[1,m'']|_a. \end{equation}
Since $|\phi^M(a)| <k_1$ for all $a \in A_B$ and $\sum_{a \in A_B}|w[1,m'']|_a = \alpha m''+o(m''),$ one has
\begin{equation}\label{eq6} 0 \leq \epsilon_1\leq k_1m''\alpha+o(m''). \end{equation}
Similarly, \begin{equation}\label{eq7}|w[1,m']|_v=|\phi^M(w[1,m''])|_v =\sum_{a \in A_U}|\phi^M(a)|_v\cdot|w[1,m'']|_a+\epsilon_2+\epsilon_3, \end{equation}
where 
\begin{equation}\label{eq8} \epsilon_2=\sum_{a \in A_B}|\phi^M(a)|_v\cdot|w[1,m'']|_a. \end{equation}
and $\epsilon_3$ is the number of occurrences of $v$ which begin in a factor $\phi^M(w[t])$ for some $t$ but do not end in the same factor. For any bounded letter $a$, $|\phi^M(a)|_v \leq |\phi^M(a)| <k_1$. Thus,
\begin{equation}\label{eq9} 0 \leq \epsilon_2\leq k_1m''\alpha+o(m''). \end{equation}
For any integer $i$, $1 \leq i \leq m''$,  the number of occurrences of $v$ which begin in $\phi^M(w[i])$ but do not end in the same factor is upper-bounded by the length of $v$, which is $L$. Thus, 
\begin{equation}\label{eq10} 0 \leq \epsilon_3\leq Lm'', \end{equation}
One can estimate $\frac{|w[1,m]|_v}{m}$ in terms of $m''$ using Equations \ref{eq4},\ref{eq7}.
\begin{equation}\label{eq11} \frac{|w[1,m]|_v}{m}=\frac{\sum_{a \in A_U}|\phi^M(a)|_v\cdot|w[1,m'']|_a+\epsilon_2+\epsilon_3+\epsilon_4}{\sum_{a \in A_U}|\phi^M(a)|\cdot|w[1,m'']|_a+\epsilon_1+\epsilon_5}, \end{equation}
where $\epsilon_4=|w[1,m]|_v-|w[1,m']|_v$ and $\epsilon_5=m-m'$. From Equations \ref{eq2}, \ref{eq3}, 
\begin{equation}\label{eq45} 0 \leq \epsilon_4, \epsilon_5 <k_2=o(m'').\end{equation} 
To simply the notation, let $$S_1(m)=\sum_{a \in A_U}|\phi^M(a)|_v\cdot|w[1,m'']|_a, \; S_2(m)=\sum_{a \in A_U}|\phi^M(a)|\cdot|w[1,m'']|_a.$$ 
Since \begin{equation} \frac{S_1(m)}{S_2(m)}= \frac{\sum_{a \in A_U}|\phi^M(a)|_v\cdot\frac{|w[1,m'']|_a}{m''}}{\sum_{a \in A_U}|\phi^M(a)|\cdot\frac{|w[1,m'']|_a}{m''}} \end{equation}
and $\lim_{m \to \infty}\frac{|w[1,m'']|_a}{m''}=\alpha_a$ for all $a \in A_U$, one has
\begin{equation}\label{eqlim} \lim_{m \to \infty}\frac{S_1(m)}{S_2(m)}= \frac{\sum_{a \in A_U}|\phi^M(a)|_v\alpha_a}{\sum_{a \in A_U}|\phi^M(a)|\alpha_a}<1. \end{equation}
Let $C_{v,M}=\frac{\sum_{a \in A_U}|\phi^M(a)|_v\alpha_a}{\sum_{a \in A_U}|\phi^M(a)|\alpha_a}$.\\

To upper bound $\frac{|w[1,m]|_v}{m}$, take $\epsilon_1=\epsilon_5=0$, $\epsilon_2=k_1m''\alpha+o(m'')$, $\epsilon_3=Lm''$ and $\epsilon_4=o(m'')$ using Equations \ref{eq6}, \ref{eq8}, \ref{eq10}, \ref{eq45}. One has
\begin{equation}\label{eq12} \frac{|w[1,m]|_v}{m}=\frac{S_1(m)+\epsilon_2+\epsilon_3+\epsilon_4}{S_2(m)+\epsilon_1+\epsilon_5}\leq \frac{S_1(m)}{S_2(m)}+\frac{(k_1\alpha+L)m''+o(m'')}{S_2(m)}. \end{equation}
Since $|\phi^M(a)| \geq \frac{6(L+\alpha k_1)}{(1-\alpha)\Delta}$ for all $a \in A_U$, 
\begin{equation}\label{eq13}\begin{aligned} \frac{(k_1\alpha+L)m''+o(m'')}{S_2(m)}&=\frac{(k_1\alpha+L)m''+o(m'')}{\sum_{a \in A_U}|\phi^M(a)|\cdot|w[1,m'']|_a}\\
&\leq \frac{(k_1\alpha+L)m''+o(m'')}{\frac{6(L+\alpha k_1)}{(1-\alpha)\Delta}((1-\alpha) m''+o(m''))}\\
&\leq \frac{\Delta}{6}+o(1).
\end{aligned} 
\end{equation}
Thus, combining Equations \ref{eq12}, \ref{eq13},
\begin{equation}\label{eq14} \frac{|w[1,m]|_v}{m}\leq \frac{S_1(m)}{S_2(m)}+\frac{\Delta}{6}+o(1). \end{equation}
From Equation \ref{eqlim}, There exists an integer $N_1$ such that for all $m \geq N_1$, \begin{equation}\label{eq15} \frac{|w[1,m]|_v}{m}< C_{v,M}+\frac{\Delta}{6}+\frac{\Delta}{6}= C_{v,M}+\frac{\Delta}{3}.\end{equation}

To lower bound $\frac{|w[1,m]|_v}{m}$, take $\epsilon_1=k_1m''\alpha+o(m'')$, $\epsilon_5=o(m'')$ and $\epsilon_2=\epsilon_3=\epsilon_4=0$ using Equations \ref{eq6}, \ref{eq8}, \ref{eq10}, \ref{eq45}. One has
\begin{equation}\label{eq16} \frac{|w[1,m]|_v}{m}=\frac{S_1(m)+\epsilon_2+\epsilon_3+ \epsilon_4}{S_2(m)+\epsilon_1+\epsilon_5}\geq \frac{S_1(m)}{S_2(m)+k_1m''\alpha+o(m'')} \end{equation}
Since $|\phi^M(a)| \geq \frac{k_1 \alpha(6-\Delta)}{(1-\alpha)\Delta}$ for all $a \in A_U$, 
\begin{equation}\label{eq17}\begin{aligned}  \frac{|w[1,m]|_v}{m}&\geq \frac{S_1(m)}{S_2(m)+k_1m''\alpha+o(m'')}\\
&\geq \frac{S_1(m)}{S_2(m)}\frac{1}{1+\frac{k_1m''\alpha+o(m'')}{S_2(m)}}\\
&\geq \frac{S_1(m)}{S_2(m)}\frac{1}{1+\frac{k_1m''\alpha+o(m'')}{\frac{k_1 \alpha(6-\Delta)}{(1-\alpha)\Delta}((1-\alpha)m''+o(m''))}}\\
&\geq \frac{S_1(m)}{S_2(m)}(1- \frac{\Delta}{6})+o(1)\\
&\geq \frac{S_1(m)}{S_2(m)}- \frac{\Delta}{6}+o(1).
\end{aligned} 
\end{equation}
The last inequality is from the fact that $\frac{S_1(m)}{S_2(m)}<1$ for all $m$. From Equation \ref{eqlim}, there exists an integer $N_2$ such that for all $m \geq N_2$, \begin{equation} \label{equp} \frac{|w[1,m]|_v}{m}> C_{v,M}-\frac{\Delta}{6}-\frac{\Delta}{6}= C_{v,M}-\frac{\Delta}{3}.\end{equation}

In conclusion, let $N=\max\{N_1, N_2\}$, then for any $m \geq N$, $ \left|\frac{|w[1,m]|_v}{m}-C_{v,M}\right|< \frac{\Delta}{3}$ which is a contradiction to the hypothesis that $\limsup\frac{|w[1,n]|_v}{n}-\liminf\frac{|w[1,n]|_v}{n}= \Delta>0$. Thus, all factors in $w$ have a frequency.   
\end{proof}

\begin{cor}
Let $w$ be a pure morphic word over the alphabet $A$ such that each element in $A$ has a frequency in $w$. If $\phi$ is the morphism such that $\phi(w)=w$, then for any $v \in A^*$, with the same notation as above,
\begin{equation}\lim_{m \to \infty}\frac{|w[1,m]|_v}{m}=\lim_{M \to \infty}C_{v,M}.\end{equation}
\end{cor}

\begin{proof}
Once it is proved that each factor $v$ of $w$ has a frequency, the argument in the proof of  Theorem \ref{main} shows that, for any $\Delta >0$, there exists $M \in \mathbb{N}^+$ such that for all $m \geq M$, there exists $N_m \in \mathbb{N}^+$ such that for all $n \geq N_m$, 
\begin{equation}\label{eq23}\left|C_{v,m}-\frac{|w[1,n]|_v}{n}\right|< \frac{\Delta}{3}.\end{equation}See Equations \ref{equp}, \ref{eq15}.  On the other hand, letting $\alpha_v$ be the frequency of $v$ in $w$, there exists $N'_m \in \mathbb{N}^+$ such that for all $n \geq N'_m$, 
\begin{equation}\left|\alpha_v-\frac{|w[1,n]|_v}{n}\right|< \frac{\Delta}{3}.\end{equation}
Thus, for $n$ large, \begin{equation}\label{lim}\left|\alpha_v-C_{v,m}\right|< \frac{2\Delta}{3}.\end{equation}
However, there is no term in Equation \ref{lim} involving the variable $n$, thus, Equation \ref{lim} holds independently from the choice of $n$. Consequently, 
\begin{equation}
\lim_{M \to \infty}C_{v,M}=\alpha_v.
\end{equation}
\end{proof}

\begin{rmk}
In the proof of Theorem \ref{main}, the existence of the frequency of the factor $v$ is not established in a direct way. The reason is that once Equation \ref{eq23} is established, it is not trivial to prove that the sequence $(C_{v,m})_{m \in \mathbb{N}^+}$ converges. However, when it is proved that the frequency of $v$ exists, then the limit of the above sequence exists as well.
\end{rmk}

\section{Conclusion}
The author concludes this note by asking the following question:
\begin{qst}
For any pure morphic word $w$ and any factor $v$ of $w$, is it possible to prove the existence of the frequency of $v$ and directly compute the value using the incidence matrix of $w$?
\end{qst}


\bibliographystyle{splncs03}
\bibliography{biblio}

\end{document}